\documentclass[12pt]{article}
\usepackage{amsmath}
\usepackage{graphicx,psfrag,epsf}
\usepackage{enumerate}
\usepackage{natbib}
\usepackage{amsmath,amsthm,amsfonts,epsfig,amssymb,natbib,eucal,eufrak}
\usepackage{booktabs}
\usepackage{multirow}
\usepackage{siunitx}
\usepackage{qtree}
\usepackage{hyperref}

\usepackage{float}
\restylefloat{table}
\usepackage{color}

\usepackage{pdfpages}
\usepackage{wrapfig}

\usepackage{comment}

\usepackage{tikz}

\newtheorem{A}{Assertion}
\newtheorem{thm}{Theorem}

\newtheorem{com}{Comment}
\newtheorem{remark}{Remark}

\newtheorem{cor}{Corollary}

\newcommand{\blind}{0}

\begin{document}

\def\spacingset#1{\renewcommand{\baselinestretch}%
{#1}\small\normalsize} \spacingset{1}


\if0\blind
{
  \title{\bf A note on the distribution of the extreme degrees of a random graph via the Stein-Chen method}

  \author{
  Yaakov Malinovsky
    \thanks{email: yaakovm@umbc.edu}
   \\
    Department of Mathematics and Statistics\\ University of Maryland, Baltimore County, Baltimore, MD 21250, USA\\
}
  \maketitle
} 

\if1\blind
{
  \bigskip
  \begin{center}
    {\LARGE\bf Title}
\end{center}
  \medskip
} \fi
\begin{abstract}

We offer an alternative proof, using the Stein-Chen method, of Bollob\'{a}s' theorem concerning the distribution of the extreme degrees of a random graph. Our proof also provides a rate of convergence of the extreme degree to its asymptotic distribution.
The same method also applies in a more general setting
where the probability of every pair of vertices being connected by edges depends on the number of vertices.
\end{abstract}

\noindent%

\noindent%

\spacingset{1.45} 

\section{Introduction}
Consider a random graph with $n$ labeled vertices $\left\{1,2,\ldots,n\right\}$ in which each edge $E_{ij}, 1\leq i<j\leq n$
is chosen independently and with a fixed probability $p$, $0<p<1$.
Denote by $d_i$ degree of the vertex $i, i=1,\ldots,n$ of a graph $G\in \mathbb{G}(n, p)$, where $\mathbb{G}(n, p)$ is the
probability space of graphs, and by $d_{1:n}\geq d_{2:n}\geq \cdots \geq d_{n:n}$ the degree sequence arranged in decreasing order.
The asymptotic distribution of $d_{n:n}$ was established in \cite{ER1961}. \cite{I1973} studied the asymptotic distribution of $d_{m:n}$ with fixed $m$ and $p=p(n)\rightarrow 0$ as $n\rightarrow \infty$.
For a fixed $p$ ($0<p<1$), \cite{B1980} found the asymptotic distribution of $d_{m:n}$ and proved:
\begin{thm}[\cite{B1980}]
\label{eq:BB}
Suppose $p$ is fixed, $0<p<1$, $q=1-p$. Then, for every fixed natural number $m$ independent of $n$, and for a fixed real number $t$,
\begin{equation}
\label{eq:B}
{\displaystyle \lim_{n\rightarrow \infty} P\left(\frac{d_{m:n}-np}{\sqrt{npq}}< a_n t+b_n\right)
=e^{-e^{-t}}\sum_{k=0}^{m-1} \frac{e^{-tk}}{k!}},
\end{equation}
where
\begin{equation*}
a_n=(2 \log n)^{-\frac{1}{2}},\,\,\,\,\, b_n=(2\log n)^{\frac{1}{2}}-\frac{1}{2}(2 \log n)^{-\frac{1}{2}}\left(\log\log n+\log 4\pi\right).
\end{equation*}
\end{thm}

We provide a proof sketch of Theorem \ref{eq:BB}.
\smallskip

\noindent
{\bf Proof sketch of \cite{B1980} }:

Let $c$ be a fixed positive constant and let $x=x(n,c)$ be defined by
\begin{equation}\label{eq:Cramer}
 n\frac{(2\pi)^{-1/2}e^{-x^2/2}}{x}=c.
\end{equation}
Set $K=pn+x(npq)^{1/2}$ and denote by $X=X(n,c)$ the number of vertices of at least $K$ degree.
For the natural number $r$, $Y_r={X \choose r}$.
Thus, Bollob\'{a}s established the inequality
\begin{equation}
\label{eq:Bi}
{n \choose r}
\left(P(Bin(n-r, p)\geq K\right)^r \leq E(Y_r)\leq {n \choose r}\left(P(Bin(n-r, p)\geq K-r-1\right)^r.
\end{equation}

Then, by applying the de Moivre-Laplace limit theorem to the left-hand side (LHS) and
right-hand side (RHS) of \eqref{eq:Bi} and using the relation
$1-\Phi(x)\sim \frac{1}{x} (2\pi)^{-1/2}e^{-x^2/2}$ as $x\rightarrow \infty$ (\cite{F1968}) p. 175, Lemma 2), where $\Phi()$ is the CDF of a standard normal random variable,
Bollob\'{a}s proved that
\begin{equation}
\label{eq:Y}
\lim_{n\rightarrow \infty}E(Y_r)=\frac{c^r}{r!}.
\end{equation}
\eqref{eq:Y} implies that for every fixed $r$, the $r$th moment of $X(n,c)$
tends to the $r$th moment of the Poisson distribution with mean $c$,
and therefore from the Carleman theorem (\cite{F1971}, pp. 227-228) follows that
\begin{equation*}
\lim_{n\rightarrow \infty}P(X(n,c)=k)=e^{-c}\frac{c^k}{k!},
\end{equation*}
i.e.,
$X(n,c)$ tends to the Poisson distribution with mean $c$ (see also \cite{B1981}).
Since for $d_{1:n}\geq d_{2:n}\geq \cdots \geq d_{n:n}$, $d_{m:n}<K$ iff $X(n,c)\leq m-1$, it follows that
\begin{equation*}
\lim_{n\rightarrow \infty}P(d_{m:n}<K)=\lim_{n\rightarrow \infty}\sum_{k=0}^{m-1}P(X(n,c)=k)=e^{-c}\sum_{k=0}^{m-1}e^{-c}\frac{c^k}{k!}.
\end{equation*}
Then, \cite{B1980} showed that when $n$ is large the solution of \eqref{eq:Cramer} is $x=a_n t+b_n,\,\,\, c=e^{-t}$.

\section{Alternative proof and Rate of Convergence}
We present an alternative proof of Bollob\'{a}s' theorem via the Stein-Chen method.
It allows us to obtain a convergence rate.

Denote normalized vertex degrees (zero expectation and unit variance) as $d_1^{*}, d_2^{*},\ldots,d_n^{*}$, i.e., $\displaystyle {d_i^{*}=(d_i-E(d_i))/(Var(d_i))^{1/2}
}$, where $E(d_i)=(n-1)p,\,\,\, Var(d_i)=(n-1)p(1-p)$ for $i=1,\ldots,n$,
and their corresponding decreasing sequence as $d_{1:n}^{*}\geq d_{2:n}^{*}\geq \cdots \geq \ldots \geq d_{n:n}^{*}$.

Let ${\displaystyle I_{i}^{(n)}=I(d_{i}^{*}>x_n(t))}$  be an indicator function of an event ${\displaystyle \left\{d_{i}^{*}>x_n(t)\right\}}$,
where we choose ${\displaystyle x_n(t)=a_n t+b_n=\sqrt{2\log n}+\frac{t-\frac{\log\log n+\log 4\pi}{2}}{\sqrt{2\log n}}}$, with $a_n$ and $b_n$ as defined in \eqref{eq:B}.

Set $${\displaystyle W_{n}=\sum_{i=1}^{n}I_{i}^{(n)},\,\,\, \pi_{i}^{(n)}=P(I_{i}^{(n)}=1),\,\,\,\lambda_n=E(W_n)=\sum_{i=1}^{n}\pi_{i}^{(n)}=n\pi_{1}^{(n)}.}$$

We will need five Assertions.
\begin{A}
\label{eq:A1}
\begin{equation*}
\pi_{1}^{(n)}=P\left(d_{1}^{*}>x_n(t)\right)\sim 1-\Phi\left(x_n(t)\right),
\end{equation*}
where the sign $\sim$ means that the ratio of the quantities on the LHS and the RHS tends to 1 for $n\rightarrow \infty$.
\end{A}

\begin{proof}
Follows from \cite{F1968}(pp. 192-193, Chapter VII.6) since for a fixed real number $t$, $x^3_n(t)/\sqrt{n}\rightarrow 0$ as $n\rightarrow \infty$.
\end{proof}

\begin{A}
\label{eq:A2}
For a fixed real number $t$,
\begin{equation*}
{\displaystyle  n\pi_{1}^{(n)}\sim e^{-t}.}
\end{equation*}
\end{A}
\begin{proof}
From  Assertion \ref{eq:A1} combined with the result on page 374 of \cite{C1946}, it follows that
${\displaystyle \lim_{n\rightarrow \infty}\lambda_n= \lim_{n\rightarrow \infty}n \pi_{1}^{(n)}=\lim_{n\rightarrow \infty}n P(d_{1}^{*}>x_n(t))=e^{-t}}.$
\end{proof}

\begin{A}
\label{eq:A3}
For a fixed real number $t$, if $npq\rightarrow \infty$ as $n\rightarrow \infty$, where $p+q=1$, then
\begin{align}
Cov\left(I_{1}^{(n)}, I_{2}^{(n)}\right)\sim
2e^{-2t} \frac{\log n}{n^{3}}.
\end{align}
\end{A}
\begin{proof}
By $B_n$ denote the random variable $Bin(n, p)$, and set $$y=(n-1)p+x_n(t)\sqrt{(n-1)pq}.$$
Under this notation, ${\displaystyle \pi_1^{(n)}=P(B_{n-1}>y)}$, and then
using the formula of total probability we obtain
\begin{align}
\label{eq:cov}
&
Cov\left(I_{1}^{(n)}, I_{2}^{(n)}\right)=
p[P(B_{n-2}>y-1)]^2+q[P(B_{n-2}>y)]^2-[pP(B_{n-2}>y-1)+qP(B_{n-2}>y)]^2 \nonumber\\
&
=
pq\left[P(B_{n-2}>y-1)-P(B_{n-2}>y)\right]^2=pq\left[P(B_{n-2}=\lfloor y\rfloor)\right]^2
=\frac{q}{p}\frac{(\lfloor y+1\rfloor)^2}{(n-1)^2}\left[P(B_{n-1}=\lfloor y+1\rfloor)\right]^2.
\end{align}

The de Moivre-Laplace theorem (\cite{R1970}, p. 204, Theorem 4.5.1.) states that uniformly in $k$,
\begin{align*}
{\displaystyle
P(B_n=k)\sim \frac{1}{\sqrt{2\pi n p q}}e^{-\frac{(k-np)^2}{2npq}}\,\,\,
\text{if}\,\,\,|k-np|=o(n^{2/3})}.
\end{align*}
Therefore, by direct calculation, it follows that for a fixed real $t$,
\begin{align}
\label{eq:Bn}
&
{\displaystyle
P(B_{n-1}=\lfloor y+1\rfloor)\sim e^{-t} \frac{\sqrt{2}}{\sqrt{pq}}\frac{\sqrt{\log n}}{n^{3/2}}}\,\,\,\,\text{if}\,\,\,\,\, {\displaystyle x_n(t)\sqrt{(n-1)pq}=o(n^{2/3})} \nonumber \\
&
\text{and}\,\,\,\, npq\rightarrow \infty\,\,\,\,\, \text{as}\,\,\, n\rightarrow \infty.
\end{align}
The condition ${\displaystyle x_n(t)\sqrt{(n-1)pq}=o(n^{2/3})}$ is satisfied since for a fixed real $t$, $x_n(t)\sim \sqrt{2\log n}$ and $pq\leq 1/4$.

Also, by direct calculation it follows that
\begin{align}
\label{eq:yn}
&
{\displaystyle \frac{\lfloor y+1\rfloor ^2}{(n-1)^2}\sim p^2.}
\end{align}
Combining \eqref{eq:cov}, \eqref{eq:Bn}, and \eqref{eq:yn}, we obtain Assertion \ref{eq:A3}.
\end{proof}

\begin{A}
\label{eq:A4}
\begin{align}
\label{eq:TV}
&
d_{TV}\left({L}(W_n), Poi(\lambda_n)\right)=\frac{1}{2}\sum_{k\geq 0}\left|P(W_n=k)-P(Poi(\lambda_n)=k)\right| \nonumber\\
&
\leq
\left(1-e^{-\lambda_n}\right)
\left[
\pi_1^{(n)}+\frac{n(n-1)Cov\left(I_{1}^{(n)}, I_{2}^{(n)}\right)}{\lambda_n}
\right]\equiv U_{TVD},
\end{align}
where ${\displaystyle d_{TV}\left({L}(W_n), Poi(\lambda_n)\right)}$ is the total variation distance (TVD) between
distributions of $W_n$ and the Poisson distribution with mean $\lambda_n$.
\end{A}

\begin{proof}
Let $e_{ij}$ be the indicator random variable for the event $\left\{E_{ij}=1\right\}$.
The indicators ${\displaystyle \left(I_{1}^{(n)}, I_{2}^{(n)},\ldots,I_{n}^{(n)}\right)}$ are increasing functions of the independent edge indicators ${\displaystyle \left\{e_{ij}, 1\leq i<j \leq n\right\} }$.  Let ${\displaystyle I_{i}^{(n)*}}$ be an independent copy of ${\displaystyle I_{i}^{(n)}}$. Thus, for all nondecreasing functions $f$ and $g$ for which expectations
(from the expansion of the RHS of the expectation below) exist,
\begin{align*}
&
2Cov\left(f(I_{1}^{(n)}), g(I_{2}^{(n)})\right)=
E\left[\left(f(I_{1}^{(n)})-f(I_{1}^{(n)*})\right)\left(g(I_{2}^{(n)})-g(I_{2}^{(n)*})\right)\right]\geq 0,
\end{align*}
i.e., ${\displaystyle \left(I_{1}^{(n)}, I_{2}^{(n)},\ldots,I_{n}^{(n)}\right)}$ are associated random variables \citep{EPW1967}.
As such, by Theorem 2.G, and hence by Corollary 2.C.4. in \cite{BHJ1992}, we obtain a particular form of an upper bound for the TVD, i.e.,
$$
{\displaystyle d_{TV}\left({L}(W_n), Poi(\lambda_n)\right)
\leq \frac{1-e^{-\lambda_n}}{\lambda_n}\left(Var(W_n)-\lambda_n+2\sum_{i=1}^{n}\left(\pi_{1}^{(n)}\right)^2\right)
}
.$$ Hence \eqref{eq:TV} follows since $d_1^{*},\ldots,d_{n}^{*}$ are identically distributed, and therefore
${\displaystyle  \sum_{i=1}^{n}\left(\pi_{i}^{(n)}\right)^2=n \left(\pi_{1}^{(n)}\right)^2,}$
${\displaystyle \sum_{i\neq j}Cov\left(I_{i}^{(n)}, I_{j}^{(n)}\right)=n(n-1)Cov\left(I_{1}^{(n)}, I_{2}^{(n)}\right)
}$.

\end{proof}

\begin{remark}
Another upper bound of ${\displaystyle d_{TV}\left({L}(W_n), Poi(\lambda_n)\right)}$ is given in Theorem 5.E in \cite{BHJ1992}.
\end{remark}

\begin{A}
\label{eq:A5}
For a fixed real number $t$, if $npq\rightarrow \infty$ as $n\rightarrow \infty$, then
\begin{equation}
\label{eq:U}
{\displaystyle U_{TVD}\sim
\frac{1+2\log n}{n}e^{-t}\left(1-e^{-e^{-t}}\right).
}
\end{equation}
\end{A}
\begin{proof}
Combining Assertions \ref{eq:A2}, \ref{eq:A3} and \ref{eq:A4}, we obtain Assertion \ref{eq:A5}.
\end{proof}

From these assertions, we further obtain the following:

\begin{thm}
\label{eq:Main}
For ${\displaystyle p \in (0, 1)}$ and a fixed value of $k$, $$\lim_{n\rightarrow \infty}P(W_n=k)=e^{-e^{-t}}\frac{e^{-tk}}{k!},$$
with a rate of convergence of $W_n$ to the Poisson distribution with mean $e^{-t}$
of order ${\displaystyle \frac{\log n}{n}}$.
\end{thm}

\begin{proof}
Fix ${\displaystyle p \in (0, 1)}$. From Assertions 4 and 5 it follows that ${\displaystyle \lim_{n\rightarrow \infty} d_{TV}\left({L}(W_n), Poi(\lambda_n)\right)=0}$,
i.e., ${\displaystyle \left\{W_n\right\}}$ converges in the TVD
to the Poisson distribution with mean $e^{-t}$.
Since  ${\displaystyle \lim_{n\rightarrow \infty}\lambda_n=e^{-t}}$ (Assertion 2), TVD convergence occurs if and only if ${\displaystyle \left\{W_n\right\}}$ in the distribution converges
to the Poisson distribution with mean $e^{-t}$ \citep{JLR2000}. The rate of convergence supplies the RHS of \eqref{eq:U}.
\end{proof}

We obtain an immediately corollary.

\begin{cor}
Suppose $p$ is fixed, $0<p<1$. Then, for a fixed natural number $m$ and a fixed real number $t$,
\begin{equation}
\label{eq:M}
{\displaystyle \lim_{n\rightarrow \infty} P\left(\frac{d_{m:n}-(n-1)p}{\sqrt{(n-1)pq}}\leq a_n t+b_n\right)
=e^{-e^{-t}}\sum_{k=0}^{m-1} \frac{e^{-tk}}{k!}}.
\end{equation}
\end{cor}

\begin{proof}
Noticing that ${\displaystyle P\left(d^{*}_{m:n}\leq x_n(t)\right)=P\left(W_n\leq m-1\right)}$,
and applying Theorem \ref{eq:Main}, we obtain \eqref{eq:M}.
\end{proof}

\begin{com}

$U_{TVD}$ can also be bounded using
(i) \eqref{eq:cov} along with the \cite{U1937} inequality (on page 135 of his book), which states that
\begin{align*}
P\left(B_n=np+w\sqrt{npq}\right)=\frac{1}{\sqrt{2\pi n p q}}e^{-w^2/2}
\left[1+\frac{(q-p)(w^3-3w)}{6\sqrt{n p q}}\right]+\triangle
\end{align*}
where
$$
|\triangle|<\frac{0.15+0.25|p-q|}{(npq)^{\frac{3}{2}}}+e^{-\frac{3}{2}\sqrt{npq}},
$$
provided $npq\geq 25$;
and using (ii) the Berry-Essen Theorem (\cite{F1971}, Chapter XVI.5) adapted to the binomial distribution,
which states that for all $w$ and $n\geq 2$,
\begin{equation}
\label{eq:BE}
\Big |P\left((B_n-np)/\sqrt{npq}>w\right)-\left(1-\Phi\left(w\right)\right)\Big|
\leq \frac{C}{\sqrt{n}}\frac{p^2+q^2}{\sqrt{pq}},
\end{equation}
where ${\displaystyle C=\frac{\sqrt{10}+3}{6\sqrt{2\pi}}=0.4968\ldots}$ \citep{S2016}.
\end{com}

\begin{com}
The method and the results can be extended to the case where $p$ depends on $n$.
In that case, Assertions 3 and 5 hold if  $p(1-p)n\rightarrow \infty$ as $n\rightarrow \infty$.
This can be compared with the condition in Theorem 3.3' of \cite{B2001}.
\end{com}

\begin{com}
If $\xi_1,\ldots,\xi_n$ are independent and identically distributed standard normal random variables,
with corresponding $\xi_{1:n}\geq \xi_{2:n}\geq \cdots \geq \xi_{n:n}$ sequence arranged in decreasing order,
then the limit distribution $P\left(\xi_{m:n} \leq a_n t+b_n\right)$
is identical to the RHS of \eqref{eq:B}, with the same $a_n, b_n$ (see for example \cite{G1987}).
However, \eqref{eq:B} is not obvious since $d_1,\ldots,d_n$
are dependent and their joint distribution depends on $n$.
\end{com}

\begin{com}
The same asymptotic distribution as in \eqref{eq:B}, for the ordered normalized scores, holds for a round-robin tournament model with $n$ players \citep{M2021}.
The difference is that the round-robin tournament is a complete directed graph and  the total scores (degrees) of the players
are negatively correlated.
\end{com}

\section*{Acknowledgement}
Research supported in part by BSF grant 2020063.

{}

\end{document}